\documentclass[12pt]{article}

\usepackage[inner=32mm,outer=31mm,tmargin=30mm,bmargin=40mm]{geometry}

\usepackage{latexsym,amsfonts,amsmath,amssymb,epsfig,tabularx,amsthm,dsfont,mathrsfs}

\usepackage{graphicx}
\usepackage{enumerate}

\usepackage{titling}
\newcommand{\subtitle}[1]{%
  \posttitle{%
    \par\end{center}
    \begin{center}\large#1\end{center}
    \vskip0.5em}%
}

\usepackage{hyperref} 
\hypersetup{colorlinks=true,
        linkcolor=black,
        citecolor=black,
        urlcolor=blue}

\usepackage{pgfplots}

\usetikzlibrary{external}
\theoremstyle{plain}

\newtheorem{thm}{Theorem}[section]
\newtheorem{lem}[thm]{Lemma}

\newtheorem{prop}[thm]{Proposition}
\newtheorem{cor}[thm]{Corollary}

\theoremstyle{definition}
\newtheorem{rem}[thm]{Remark}
\newtheorem{example}[thm]{Example}

\newtheorem{defi}[thm]{Definition}
\newtheorem{alg}[thm]{Algorithm}

\DeclareMathOperator{\supp}{supp}

\renewcommand{\rho}{\varrho}

\DeclareMathAlphabet{\mathup}{OT1}{\familydefault}{m}{n}

\newcommand{\widebar}[1]{\mbox{\kern1.5pt\hbox{\vbox{\hrule height 0.6pt \kern0.35ex
        \hbox{\kern-0.15em \ensuremath{#1 }\kern0.0em}}}}\kern-0.1pt}

\newlength{\fixboxwidth}
\usepackage{color}
\usepackage{soul}

\usepackage{latexsym,amsfonts,amsmath,amssymb,mathrsfs}
\usepackage{url}
\usepackage{graphicx}
\usepackage{color}
\usepackage{euscript}
\usepackage{verbatim}
\usepackage{hyperref}

\begin{document}

\title{Quantitative spectral gap estimate and Wasserstein contraction of simple slice sampling}

\author{Viacheslav Natarovskii\thanks{Institute for Mathematical Stochastics, 
Georg-August-Universit\"at G\"ottingen, Goldschmidtstra\ss e 7, 37077 G\"ottingen, 
Email: vnataro@uni-goettingen.de, daniel.rudolf@uni-goettingen.de}, 
Daniel Rudolf$^{\,*,}$\thanks{Felix-Bernstein-Institute for Mathematical Statistics 
 in the Biosciences, Goldschmidtstra\ss e 7, 37077 G\"ottingen}, 
 Bj\"orn Sprungk\thanks{Faculty of Mathematics and Computer Science, Technische Universit\"at Bergakademie Freiberg, bjoern.sprungk@math.tu-freiberg.de}}

\date{\today}

\maketitle

\begin{abstract}
We prove Wasserstein contraction of simple slice sampling for approximate sampling w.r.t. distributions with log-concave and rotational invariant Lebesgue densities.
This yields, in particular, an explicit quantitative lower bound of the spectral gap of simple slice sampling. Moreover, this lower bound carries over to more general target distributions depending only on the volume of the (super-)level sets of their unnormalized density.
\end{abstract}


{\bf Classification. Primary: 65C40; Secondary: 60J22, 62D99, 65C05.} \\
\\

\section{Introduction}
A challenging problem in Bayesian statistics and computational science is sampling w.r.t. distributions which are only known up to a normalizing constant. Assume that $G\subseteq\mathbb{R}^{d}$ and $\rho:G\rightarrow(0,\infty)$
is integrable w.r.t.~to the Lebesgue measure. The goal is to sample w.r.t.~the distribution 
determined by $\rho$, say $\pi$, that is,
\[
\pi(A)=\frac{\int_{A}\rho(x){\rm d}x}{\int_{G}\rho(x){\rm d}x},\qquad A\in\mathcal{B}(G).
\]
Here $\mathcal{B}(G)$ denotes the Borel $\sigma$-algebra. In most cases this can only be done approximately and the idea is to construct a (time-homogeneous) Markov chain $(X_n)_{n\in\mathbb{N}}$ which has $\pi$ as limit distribution, i.e., for increasing $n$ the distribution of $X_n$ converges to $\pi$. Slice sampling methods provide auxiliary variable Markov chains for doing this and several different versions have been proposed and investigated \cite{besag1993spatial,Hi98,latuszynski2014convergence,mira2001perfect,MiTi02,MuAdMa10,Ne03,RoRo99,RoRo02}. In particular also Metropolis-Hastings algorithms can be considered as such methods, see \cite{Hi98,RuUl15}. 
In the underlying work we investigate simple slice sampling which works as follows\footnote{It is straightforward to verify that $\pi$ is a stationary  distribution of the simple slice sampler.}: 
\begin{alg} \label{alg:simple_slice}
	Given the current state  $X_n=x\in G$ the simple slice sampling algorithm generates the next Markov chain
	instance $X_{n+1}$ by the following two steps: 
	\begin{enumerate}
		\item Draw $T_n$ uniformly distributed in $[0,\rho(x)]$, call the result $t$.
		\item Draw $X_{n+1}$ uniformly distributed on 
		\[
		G(t) := \{x\in G\mid \rho(x)\geq t\},
		\]
		the (super-) level set of $\rho$ at $t$.
	\end{enumerate}
\end{alg}
The charm of this algorithmic approach 
lies certainly in the empirically attestable and intuitively reasonable well-behaving convergence properties of the corresponding Markov chain. Indeed, robust convergence properties are also established theoretically. Mira and Tierney in \cite{MiTi02} prove uniform ergodicity under boundedness conditions on $G$ and $\rho$. Roberts and Rosenthal \cite{RoRo99} provide qualitative 
statements about geometric ergodicity under weak assumptions as well as prove
quantitative estimates of the total variation distance of the difference of the distribution of $X_n$ and $\pi$ 
under a condition on the initial state. However, less is known about the spectral gap. 
Namely, beyond the general implications \cite{RoRo97,Ru12} from uniform and geometric ergodicity 
of the results of \cite{MiTi02,RoRo99} there is, to our knowledge, no explicit estimate of the spectral gap of simple slice sampling available.
Let $U_\rho$ be the transition operator/kernel of a Markov chain generated by simple slice sampling of a distribution $\pi$ with (unnormalized) density $\rho$. The spectral gap is defined by
\[
{\rm gap}_\pi(U_\rho):= 1 - \Vert U_\rho \Vert_{L^0_2(\pi) \to L^0_2(\pi)},
\]
where $L^0_2(\pi)$ is the space of functions $f\colon G\to \mathbb{R}$ with zero mean and finite variance  
(i.e., $\mathbb{E}_\pi(f):=\int_G f {\rm d} \pi=0$; $\Vert f \Vert_{2,\pi}^2:= \int_G \vert f \vert^2 {\rm d} \pi <\infty$). 
A spectral gap, that is, ${\rm gap}_\pi(U_\rho)>0$, leads to desirable robustness and convergence properties. For example, it is well known that a spectral gap implies geometric ergodicity \cite{kontoyiannis2009geometric,RoRo97}, and since $U_\rho$ is reversible, it also implies a central limit theorem (CLT) 
for all $f\in L_2(\pi)$, see \cite{kipnis1986central}. In addition to that it allows the estimation of the CLT asymptotic variance \cite{flegal2010batch}. In particular, an explicit lower bound of ${\rm gap_{\pi}}(U_\rho)$ leads to quantitative estimates of the total variation distance and a mean squared error bound of Markov chain Monte Carlo. More precisely, it is well known, see for instance \cite[Lemma~2]{NoRu14}, that 
\[
\Vert \nu U_\rho^n -\pi \Vert_{\rm tv} \leq (1-{\rm gap}_\pi(U_\rho))^n \left\Vert \frac{{\rm d}\nu}{{\rm d}\pi}-1 \right\Vert_{2,\pi},
\]
where $\Vert \nu - \mu \Vert_{\rm tv} := \sup_{A \in \mathcal B(G)} |\nu(A)-\mu(A)|$ denotes the total variation distance,
$\nu=\mathbb{P}_{X_1}$ and $\nu U_\rho^n=\mathbb{P}_{X_{n+1}}$. Moreover, in \cite{Ru12} it is shown for the sample average 
that  
\[
\mathbb{E} \left \vert \frac{1}{n} \sum_{j=1}^n f(X_j) -\mathbb{E}_\pi(f) \right \vert^2 
\leq \frac{2}{n\cdot {\rm gap}_\pi(U_\rho)} + \frac{c_p \left \Vert \frac{{\rm d}\nu}{{\rm d}\pi} -1 \right \Vert_\infty}{n^2 \cdot {\rm gap}_\pi(U_\rho)},
\]
for any $p>2$ and any $f\colon G\to \mathbb{R}$ with $\Vert f \Vert_p^p = \int_G \vert f \vert^p {\rm d}\pi \leq 1$, where $c_p$ is an explicit constant which depends only on $p$.

The crucial drawback of simple slice sampling is that the second step in the algorithm is difficult to perform, in particular, in high-dimensional scenarios. However, in \cite{Ne03} and the more recent papers \cite{muller2013slice,MuAdMa10,nishihara2014parallel,tibbits2014automated,tibbits2011parallel} efficient slice sampling algorithms are designed, which mimic (to some extent) simple slice sampling. Already \cite{Ne03} constructs a number of algorithms which perform a single Markov chain step on the chosen level set instead of sampling the uniform distribution. We call those methods hybrid slice sampler. For us the motivation to study simple slice sampling is twofold:
\begin{enumerate}
	\item There is to our knowledge no quantitative statement about the spectral gap available and for simple slice sampling one would expect particularly good dependence on the dimension which we to some extent verify.
	\item In the recent work of \cite{latuszynski2014convergence} it is proven that certain hybrid slice sampler, in terms of spectral gap, are, on the one hand, worse than simple slice sampling but on the other hand not much worse. 
	Hence knowledge of the spectral gap of simple slice sampling might carry over to estimates of the spectral gap of hybrid slice samplers, in particular to those suggested in \cite{Ne03}. 
\end{enumerate}

Now let us explain the main results of the underlying work. For this let the Wasserstein distance w.r.t.~the Euclidean norm $\vert \cdot \vert$ of probability measures $\nu,\mu$ on $(G,\mathcal{B}(G))$ be given by 
\[
W(\mu,\nu):=\inf_{\gamma\in\Gamma(\mu,\nu)}\int_{G\times G}\vert x-y \vert\, {\rm d}\gamma(x,y),
\]
where $\Gamma(\mu,\nu)$ is the set of all couplings of $\mu$ and
$\nu$. The set of couplings is defined by all measures on $G\times G$ with marginals
$\mu$ and $\nu$. 

\textbf{First main result (Theorem~\ref{thm:WD_log-concave}):}
For a rotational invariant and log-concave (unnormalized) density $\rho$ defined either on Euclidean balls or the whole $\mathbb{R}^d$ 
we show in Theorem~\ref{thm:WD_log-concave} 
Wasserstein contraction of simple slice sampling, that is, for all $x,y\in G\subseteq \mathbb{R}^d$  we have 
\[
W(U_{\rho}(x,\cdot),U_\rho(y,\cdot)) \leq \left(1-\frac{1}{d+1}\right)\; \vert x-y\vert.
\]
This has a number of useful consequences. It is well known, see for instance \cite[Section~2]{RuSc18}, that this implies
\begin{equation}
\label{eq:mixing}
W(\nu U_\rho^n,\pi) \leq \left(1-\frac{1}{d+1}\right)^n W(\nu,\pi)
\end{equation}
for any initial distribution $\nu$ on $G$. In addition to that by \cite[Theorem~1.5]{ChWa94}, see also \cite[Proposition~30]{Ol09}, it implies ${\rm gap}_\pi(U_\rho) \geq 1/(d+1)$. Two simple examples which satisfy the assumptions of Theorem~\ref{thm:WD_log-concave} are given by 
$\rho(x)=\exp(-\vert x \vert)$ and $\rho(x)=\exp(-\vert x \vert ^2/2)$ where $G=\mathbb{R}^d$. 
For the former one Roberts and Rosenthal in \cite{RoRo02} argue with empirical experiments that simple slice sampling ``does not mix rapidly in higher dimensions''. 
Indeed, we observe theoretically that for increasing dimension the performance of simple slice sampling gets worse, however, we disagree to some extent to their statement, since the dependence on the dimension is moderate. Namely, from \eqref{eq:mixing} we obtain for any initial  distribution that for
$W(\nu U_\rho^n,\pi)\leq \varepsilon$ with $\varepsilon\in(0,1)$ we need
\[
n	\geq 	(d+1)\log(\varepsilon^{-1}W(\nu,\pi)),
\] 
which increases only linearly in $d$.

\textbf{Second main result (Theorem~\ref{thm:exp}):}
Based on the fact that in the second step of Algorithm~\ref{alg:simple_slice} we sample w.r.t.~the uniform distribution on the (super-)level set $G(t)$, one can conjecture that its geometric shape does not matter. 
However, its ``size'' or volume should matter\footnote{This is already observed in \cite{RoRo99,RoRo02}.}. 
To this end, we define the level-set function $\ell_\rho\colon(0,\infty) \to [0,\infty)$ of $\rho\colon G\to (0,\infty)$, with $G\subseteq \mathbb{R}^d$, by $\ell_\rho(t):=\lambda_d(G(t))$ for $t\in(0,\infty)$, where $\lambda_d$ denotes the $d$-dimensional Lebesgue measure. 
The idea is now, to identify certain ``nice'' properties of $\ell_\rho$ which lead to spectral gap estimates. 
Here, we propose classes $\Lambda_k$, with $k\in\mathbb N$, of level-set functions containing all continuous $\ell\colon (0,\infty)\to [0,\infty)$ satisfying, that
\begin{itemize}
	\item $\ell$ is strictly decreasing on the open interval 
	\[
	\supp \ell := (0,\sup\{t\in(0,\infty)\mid \ell(t)>0\})
	\]
	(which implies the existence of the inverse $\ell^{-1}$ on $(0,\left \Vert \ell \right \Vert_{\infty})$ with 
	$\left \Vert \ell \right \Vert_{\infty}:=\sup_{s\in(0,\infty)} \ell(s)$), and
	\item the function $g\colon (0,\left \Vert \ell \right \Vert_{\infty}^{1/k}) \to \supp \ell$, given by $g(s)=\ell^{-1}(s^k)$ is log-concave (i.e., $\log g$ is concave).
\end{itemize}
In Theorem~\ref{thm:exp} we then show that, if for an unnormalized density $\rho\colon G\to (0,\infty)$ we have $\ell_\rho\in \Lambda_k$ for a $k \in \mathbb N$, then
\begin{equation}  \label{eq:specgapest}
{\rm gap}_\pi(U_\rho) \geq \frac{1}{k+1}.
\end{equation}
A crucial tool in the proof of Theorem~\ref{thm:exp} is the equality of the spectral gap of $U_\rho$ and the spectral gap of the transition operator of the ``level Markov chain'' $(T_n)_{n\in\mathbb{N}}$ defined within Algorithm~\ref{alg:simple_slice}. This statement is provided in Lemma~\ref{lem:equal_spectral_gap}.
Observe, that in the formulation of the second main result we did not impose any uni-modality, log-concavity or rotational invariance assumption on $\rho$. 
It is allowed that the $d$-variate function $\rho$ has more than one mode, the only requirement is that the corresponding level-set function belongs to $\Lambda_k$. 
In many cases, for ${k=d}$ this is satisfied, however, also $k<d$ is possible, see Example~\ref{ex:exponential}. It contains the special case where $\rho$ is assumed to be the density of the $d$-variate standard normal distribution, which leads to $\ell_\rho \in \Lambda_{\lfloor d/2 \rfloor}$. In that case for large $d$ the lower bound from \eqref{eq:specgapest} improves the spectral gap estimate of Theorem~\ref{thm:WD_log-concave} roughly by a factor of $2$. We also consider a $d$-variate ``volcano density'', where we show that this leads to a level-set function in $\Lambda_1$, such that the corresponding spectral gap of simple slice sampling is independent of the dimension satisfying the lower bound $1/2$.

The outline of the paper is as follows.
In the next section we provide 
basic notation and prove our main result w.r.t.~the Wasserstein contractivity. Then, in Section~\ref{sec:Spectral-gap} we state and discuss the necessary operator theoretic definitions and investigate the important relation between the Markov chains $(X_n)_{n\in\mathbb N}$ and $(T_n)_{n\in\mathbb N}$ generated by the simple slice sampling algorithm.
There we also prove the main theorem about the lower bound of the spectral gap and illustrate the result after a discussion about the sets $\Lambda_k$ by examples. 

\section{Wasserstein contraction}

Let $(\Omega,\mathcal{F},\mathbb{P})$ be the common probability space on which all random variables are defined. The sequence of random variables $(X_n)_{n\in\mathbb{N}}$ determined by Algorithm~\ref{alg:simple_slice} provides a Markov chain on $G$, that is, for all $A\in \mathcal{B}(G)$ it satisfies (almost surely)
\[
\mathbb{P}(X_{n+1}\in A\mid X_1,\dots,X_n) = U_\rho(X_n,A),
\]
where the transition kernel of simple slice sampling $U_\rho \colon G\times \mathcal{B}(G) \to [0,1]$ is given by
\[
U_{\rho}(x,A)=\frac{1}{\rho(x)}\int_{0}^{\rho(x)}U_{t}(A)\,{\rm d}t.
\]
Here $U_t$ denotes the uniform distribution on the level set 
\[
G(t) := \{ x\in \mathbb{R}^d \mid \rho(x)\geq t \},
\]
thus, $U_t(A) = \frac{\lambda_d(A\cap G(t)}{\lambda_d(G(t))}$ for $t>0$. Note that by construction the transition kernel $U_\rho$ is reversible w.r.t.~$\pi$, that is,
\[
\int_{B}U_{\rho}(x,A)\pi({\rm d}x)=\int_{A}U_{\rho}(x,B)\pi({\rm d}x),\quad A,B\in\mathcal{B}(G).
\]
In particular, this implies that $\pi$ is a stationary distribution of $U_\rho$. Further, by $B^{(d)}_R$ we denote the
$d$-dimensional closed Euclidean ball with radius $R>0$ around zero and by $\mathring{B}^{(d)}_R$ its interior. For log-concave rotational invariant unnormalized densities we formulate now our Wasserstein contraction result of the simple slice sampler.

\begin{thm}  
	\label{thm:WD_log-concave}
	For $R\in(0,\infty]$ let $\varphi \colon [0,R)\to \mathbb{R}$ be a strictly increasing and convex function on $[0,R)$. Define $\rho\colon \mathring{B}^{(d)}_R \to (0,\infty)$ by $\rho(x):= \exp\left( -\varphi(\vert x \vert ) \right)$.
	Then, for any ${x,y\in \mathring{B}^{(d)}_R}$ we have
	\begin{equation} \label{eq:Wasscontr}
	W(U_{\rho}(x,\cdot),U_{\rho}(y,\cdot))\leq\left(1-\frac{1}{d+1}\right) \big\vert \left \vert x\right \vert -\left \vert y\right \vert\big \vert.
	\end{equation}
\end{thm}
Before we prove the result let us provide some comments on it. 
\begin{rem}
	Let us emphasize here that we allow $R=\infty$, which leads to $\mathring{B}_R=\mathbb{R}^d$. 
	Moreover, we remark that since on the right-hand side of \eqref{eq:Wasscontr} we have the absolute value of the difference of the Euclidean norm of $x$ and $y$ an immediate consequence by the triangle inequality is
	\[
	W(U_{\rho}(x,\cdot),U_{\rho}(y,\cdot))\leq\left(1-\frac{1}{d+1}\right)\left\vert x - y\right \vert.
	\]
\end{rem}

\begin{example}  \label{ex:normal_contr}
	Let $\varphi \colon [0,\infty) \to \mathbb{R}$ be given as $\varphi(s)=s^2/2$. 
	This gives $\rho(x) = \exp(-\vert x \vert^2/2)$ which leads to $\pi$ being a multivariate standard normal density. 
	With 
	$R=\infty$ and the convexity of $\varphi$ we obtain \eqref{eq:Wasscontr}.
\end{example}

For the proof of Theorem~\ref{thm:WD_log-concave} we need the following auxiliary result. 

\begin{lem}
	\label{lema:WD-1}
	With $G=\mathring{B}^{(d)}_R$ let $\rho\colon G \to (0,\infty)$ be given as in Theorem~\ref{thm:WD_log-concave}. 
	Then, for any $x,y\in G$ we have
	\[
	W(U_{\rho}(x,\cdot),U_{\rho}(y,\cdot))\leq \frac{d}{d+1}\cdot\frac{1}{\lambda_{d}\big(B^{(d)}_1\big)^{1/d}}
	\int_{0}^{1}\left|\ell_\rho(r \rho(x))^{1/d}-\ell_\rho(r \rho(y))^{1/d}\right|{\rm d}r,
	\]
	where $\ell_\rho \colon (0,\infty) \to [0,\infty)$ is the level-set function defined by $\ell_\rho(t):=\lambda_{d}\left(G(t)\right)$.
\end{lem}
\begin{proof}
	Since $\varphi$ is strictly increasing and convex 
	it is continuous and thus injective.  
	Moreover, note that the image of $\varphi$ satisfies $\varphi([0,R))=[-\log\|\rho\|_{\infty},-\log\inf\rho)$. Here $\Vert \rho \Vert_\infty := \sup_{x\in \mathring{B}^{(d)}_R} \rho(x)$ and $\inf\rho $ is an abbreviation of $\inf_{x\in \mathring{B}^{(d)}_R} \rho(x)$ with the convention $\log 0:= -\infty$.
	Hence, there exists the inverse  
	\[
	\varphi^{-1}\colon[-\log\|\rho\|_{\infty},-\log\inf\rho)\to [0,R).
	\]
	In the case $\inf\rho = 0$ the inverse $\varphi^{-1}$ is defined on $[-\log\|\rho\|_{\infty},\infty)$.
	In the case $\inf\rho>0$ we extend the inverse $\varphi^{-1}$ to $[-\log\|\rho\|_{\infty},\infty)$ by setting 
	\[
	\varphi^{-1}(t):=\sup\left\{ s\in[0,R)\colon\varphi(s)\leq t\right\} ,\qquad t\in[-\log\|\rho\|_{\infty},\infty).
	\]
	Note that by this extension we do not change $\varphi^{-1}$ in $[-\log\|\rho\|_{\infty},-\log\inf\rho)$
	and obtain 
	\[
	\varphi^{-1}(t)=R\qquad\forall 
	t\geq-\log\inf \rho.
	\]
	For simplicity of the notation we write $\ell$
	for $\ell_\rho$.
	Observe that
	\[
	G(t) = \{x\in  \mathring{B}^{(d)}_R \mid \vert x \vert \leq \varphi^{-1}(\log t^{-1}) \}
	= B_{\left(\ell(t)/\lambda_{d}(B_1^{(d)})\right)^{1/d}}^{(d)}, \qquad t\in(0,\Vert \rho \Vert_\infty),
	\]
	since $\ell(t) = \lambda_{d}(G(t))= \varphi^{-1}(\log t^{-1})^d\ \lambda_d( B_1^{(d)} )$. Thus, $U_t$ denotes the uniform distribution on the Euclidean ball around the origin with radius $\left(\ell(t)/\lambda_{d}(B_1^{(d)})\right)^{1/d}$.
	Now it is straightforward to verify that $u_{t,s}\colon \mathcal{B}(G^2) \to [0,1]$ determined by
	\[
	u_{t,s}(A\times B):=\frac{1}{\lambda_{d}\big(B^{(d)}_1\big)}\int_{B^{(d)}_1}\mathbf{1}_{A}\bigg(\bigg(\frac{\ell(t)}{\lambda_{d}(B^{(d)}_1)}\bigg)^{1/d}z\bigg)\mathbf{1}_{B}\bigg(\bigg(\frac{\ell(s)}{\lambda_{d}(B^{(d)}_1)}\bigg)^{1/d}z\bigg){\rm d}z,
	\]
	where $A,B\in \mathcal{B}(G)$, is a coupling of $U_{t}$ and $U_{s}$. For example, we have
	\begin{align*}
	u_{t,s}(A\times G) & = \frac{1}{\lambda_{d}\big(B^{(d)}_1\big)}\int_{B^{(d)}_1}\mathbf{1}_{A}\bigg(\bigg(\frac{\ell(t)}{\lambda_{d}(B^{(d)}_1)}\bigg)^{1/d}z\bigg) {\rm d}z\\
	& = \frac{1}{\ell(t)} \int_{G(t)} \mathbf{1}_A(y) {\rm d } y = U_t(A).
	\end{align*}
	Further, note that $c\colon G^2 \times \mathcal{B}(G^2) \to [0,1]$ 
	determined by
	\begin{align*}
	c(x,y,A\times B) & :=\int_{0}^{1}u_{r\rho(x),r\rho(y)}(A\times B){\rm d}r
	\end{align*}
	is a Markovian coupling of $U_{\rho}(x,\cdot)$ and $U_{\rho}(y,\cdot)$, i.e.,
	$c(x,y,A\times G)=U_\rho(x,A)$ and $c(x,y,G\times B)=U_\rho(y,B)$ for all $x,y\in G$ and $A,B\in \mathcal{B}(G)$.
	Indeed, since
	\[
	u_{t,s}(A\times G)=U_{t}(A),\qquad u_{t,s}(G\times B)=U_{s}(B)
	\]
	we get for example
	\[
	c(x,y,A\times G)=\int_{0}^{1}U_{r\rho(x)}(A){\rm d}r=\frac{1}{\rho(x)}\int_{0}^{\rho(x)}U_{t}(A){\rm d}t=U_{\rho}(x,A).
	\]
	Summarized, for arbitrary $x,\widetilde{x}\in G$ and $A,B\in \mathcal{B}(G)$ we obtain
	\[
	c(x,\widetilde{x},A \times B)=\frac{1}{\lambda_{d}\big(B^{(d)}_1\big)}\int_{0}^{1}\int_{B^{(d)}_1} \mathbf{1}_{A}\bigg(\bigg(\frac{\ell(r\rho(x))}{\lambda_{d}(B^{(d)}_1)}\bigg)^{1/d}z\bigg)\mathbf{1}_{B}\bigg(\bigg(\frac{\ell(r\rho(\widetilde{x}))}{\lambda_{d}(B^{(d)}_1)}\bigg)^{1/d}z\bigg){\rm d}z{\rm d}r.
	\]
	Using the Markovian coupling we obtain for arbitrary $x,\widetilde{x}\in G$ that
	\begin{align*}
	W(U_{\rho}(x,\cdot),U_{\rho}&(\widetilde{x},\cdot))  \leq\int_{G^{2}}\left|y-\widetilde{y}\right|c(x,\widetilde{x}, {\rm d}y\, {\rm d}\widetilde{y} )\\
	& =\frac{1}{\lambda_{d}(B^{(d)}_1)}\int_{0}^{1}\int_{B^{(d)}_1}\left|\left(\frac{\ell(r\rho(x))}{\lambda_{d}(B^{(d)}_1)}\right)^{1/d}-\left(\frac{\ell(r\rho(\widetilde{x}))}{\lambda_{d}(B^{(d)}_1)}\right)^{1/d}\right||z|{\rm d}z{\rm d}r\\
	& =\frac{\lambda_{d}\big(B^{(d)}_1\big)}{\lambda_{d}(B^{(d)}_1)^{1+1/d}}\cdot \frac{d}{d+1}\int_{0}^{1}\left|\ell(r\rho(\widetilde{x}))^{1/d}-\ell(r\rho(x))^{1/d}\right|{\rm d}r,
	\end{align*}
	which finishes the proof.
\end{proof}

\begin{rem}
	\label{rem: opt_coupl_ust}
	In the previous proof we used the coupling $u_{t,s} \in \Gamma(U_t,U_s)$ for $s,t\in (0,\Vert \rho \Vert_\infty)$. 
	In the setting of Lemma~\ref{lema:WD-1} observe that for $d=1$ it 
	is related to the optimal Hoeffding-Fr{\'e}chet coupling. This optimality property also holds for arbitrary $d>1$, which is justified as follows.
	We derive an upper bound for $W(U_t,U_s)$ by $u_{t,s}$,
	\begin{align*}
	W(U_t,U_s) & \leq \int_{G\times G} \vert x-y \vert\, {\rm d}u_{t,s}(x,y)\\
	&= \left \vert \left(\frac{\ell(t)}{\lambda_{d}(B^{(d)}_1)}\right)^{1/d} 
	- \left(\frac{\ell(s)}{\lambda_{d}(B^{(d)}_1)}\right)^{1/d}
	\right \vert \int_{B^{(d)}_{1}}|z| \frac{\textrm{d}z}{\lambda_d\big(B^{(d)}_{1}\big)}\\
	& = \left \vert \left(\frac{\ell(t)}{\lambda_{d}(B^{(d)}_1)}\right)^{1/d} 
	- \left(\frac{\ell(s)}{\lambda_{d}(B^{(d)}_1)}\right)^{1/d}
	\right \vert \frac{d}{d+1},
	\end{align*}
	where we used $\int_{B^{(d)}_{1}}|z|\textrm{d}z=\frac{d}{d+1}\lambda_d\big(B^{(d)}_{1}\big).$
	To derive a lower bound of $W(U_t,U_s)$ we apply the Kantorovich-Rubinstein duality formula of the Wasserstein distance (see e.g. \cite[Chapter~1.2]{Vi03},) w.r.t.~$U_t$ and $U_s$. 
	It is given by
	\[
	W(U_t,U_s) = \sup_{\left\Vert g\right\Vert _{{\rm Lip}}\leq1}
	\left|\int_{G} g(z)\left(U_{t}({\rm d}z)-U_{s}({\rm d}z)\right)\right|,
	\]
	where $\Vert g \Vert_{\rm Lip} := \sup_{x,y\in G} \frac{\vert g(x)-g(y) \vert}{\vert x- y \vert }$ for $g\colon G \to \mathbb{R}$. (The supremum is 
	taken over Lipschitz continuous functions with Lipschitz constant less or equal to $1$.) 
	Considering $h(z):= \vert z \vert$ and noting $\Vert h \Vert_{\rm Lip} \leq 1$ as well as
	\[
	\int_{G} \vert z\vert\, U_{t}({\rm d}z) = \left(\frac{\ell(t)}{\lambda_{d}(B^{(d)}_1)}\right)^{1/d} \int_{B^{(d)}_{1}}|z| \frac{\textrm{d}z}{\lambda_d\big(B^{(d)}_{1}\big)} 
	= \left(\frac{\ell(t)}{\lambda_{d}(B^{(d)}_1)}\right)^{1/d} \frac{d}{d+1}	 
	\]
	then yields
	\begin{align*}
	W(U_t,U_s) 
	& \geq \left \vert \int_G h(z) (U_t({\rm d}z)) - U_s({\rm d}z))  \right \vert
	=\left \vert \left(\frac{\ell(t)}{\lambda_{d}(B^{(d)}_1)}\right)^{1/d} 
	- \left(\frac{\ell(s)}{\lambda_{d}(B^{(d)}_1)}\right)^{1/d}
	\right \vert \frac{d}{d+1}.	 
	\end{align*}
	Hence
	\[
	W(U_t,U_s) = \left \vert \left(\frac{\ell(t)}{\lambda_{d}(B^{(d)}_1)}\right)^{1/d} 
	- \left(\frac{\ell(s)}{\lambda_{d}(B^{(d)}_1)}\right)^{1/d}
	\right \vert \frac{d}{d+1},
	\]
	which implies that $u_{t,s}$ is an optimal coupling.
\end{rem}

Now we provide the proof of Theorem~\ref{thm:WD_log-concave}.
\begin{proof}[Proof of Theorem~\ref{thm:WD_log-concave}]
	Again, for $\ell_\rho$ we write  $\ell$.
	To verify the claim of the theorem by Lemma~\ref{lema:WD-1} it is sufficient to show that 
	\[
	\frac{1}{\lambda_{d}(B^{(d)}_1)^{1/d}}\int_{0}^{1}\left|\ell(r\rho(x))^{1/d}-\ell(r\rho(y))^{1/d}\right|{\rm d}r\leq\big\vert \left \vert x\right \vert -\left \vert y\right \vert \big \vert,
	\quad \forall x,y\in \mathring{B}^{(d)}_R.
	\]
	Then, by the extended inverse $\varphi^{-1}$ derived in the proof of Lemma~\ref{lema:WD-1} we have
	\begin{equation} \label{eq: ellrhophi}
	\ell(t)=\lambda_{d}(B^{(d)}_1)(\varphi^{-1}(-\log t))^{d},\qquad t\in(0,\|\rho\|_{\infty}].
	\end{equation}
	Here also note that by the definition of $\rho$ we have $\varphi(0)=-\log\|\rho\|_{\infty}$.
	The representation \eqref{eq: ellrhophi} yields for any $r\in(0,1]$ and $x\in\mathring{B}^{(d)}_R$
	that 
	\begin{align*}
	\ell(r\rho(x))^{1/d} & =\lambda_{d}\big(B^{(d)}_1\big)^{1/d}\varphi^{-1}(\varphi(|x|)-\log r),
	\end{align*}
	which leads to
	\begin{align*}
	& \quad \lambda_{d}(B^{(d)}_1)^{-1/d}\int_{0}^{1}\left|\ell(r\rho(x))^{1/d}-\ell(r\rho(y))^{1/d}\right|\ {\rm d}r \\
	& =\int_{0}^{1}\left|\varphi^{-1}(\varphi(|x|)-\log r)-\varphi^{-1}(\varphi(|y|)-\log r)\right|\ {\rm d}r.
	\end{align*}
	We now show that for any $r\in(0,1]$ and any $s,\widetilde{s}\in[0,R)$ we have
	\[
	\left|\varphi^{-1}(\varphi(s)-\log r)-\varphi^{-1}(\varphi(\widetilde{s})-\log r)\right|\leq|s-\widetilde{s}|,
	\]
	which immediately yields the assertion 
	of the theorem. 
	
	For this let $s,\widetilde{s}\in [0,R)$ and assume without loss of generality that $s\leq\widetilde{s}$.
	Define for arbitrary fixed $s\in[0,R)$ the value 
	$r_{\min}(s)$ by 
	\[
	\varphi(s)-\log r_{\min}(s)=-\log\inf\rho.
	\]
	Hence
	\[
	\varphi^{-1}(\varphi(s)-\log r)=R,\qquad\forall r\leq r_{\min}(s).
	\]
	Moreover, we set
	\[
	s'(r) := \varphi^{-1}(\varphi(s)-\log r) \in [0,R), \qquad \forall r>r_{\min}(s)
	\]
	and since $\varphi$ is continuous and increasing we have 
	\[
	\varphi(s'(r))=\varphi(s)-\log r\geq\varphi(s),
	\qquad
	s\leq s'(r).
	\]
	The same arguments lead to
	\[
	\varphi^{-1}(\varphi(\widetilde{s})-\log r)=R,\qquad\forall r\leq r_{\min}(\widetilde{s})
	\]
	and 
	\[
	\varphi(\widetilde{s}'(r))=\varphi(\widetilde{s})-\log r\geq\varphi(\widetilde{s}),
	\qquad
	\widetilde{s}\leq\widetilde{s}'(r)
	\]
	for
	\[
	\widetilde{s}'(r) := \varphi^{-1}(\varphi(\widetilde{s})-\log r) \in [0,R),\qquad\forall r>r_{\min}(\widetilde{s}).
	\]
	Note, that due to $s\leq\widetilde{s}$ we have $\varphi(s)\leq\varphi(\widetilde{s})$
	and, thus, $r_{\min}(\widetilde{s})\leq r_{\min}(s)$. We distinguish three cases w.r.t.~$r\in(0,1]$: 
	\begin{enumerate}
		\item Assume $r\leq r_{\min}(\widetilde{s})$: Here $\varphi^{-1}(\varphi(s)-\log r)=\varphi^{-1}(\varphi(\widetilde{s})-\log r)=R$
		and 
		\[
		0=\left|\varphi^{-1}(\varphi(s)-\log r)-\varphi^{-1}(\varphi(\widetilde{s})-\log r)\right|\leq|s-\widetilde{s}|.
		\]
		\item Assume $r>r_{\min}(s)$: Here
		\[
		\left|\varphi^{-1}(\varphi(s)-\log r)-\varphi^{-1}(\varphi(\widetilde{s})-\log r)\right|=|s'(r)-\widetilde{s}'(r)|
		\]
		with $s'(r),\widetilde{s}'(r)\in [0,R)$. 
		We now exploit the convexity of $\varphi$ on $[0,R)$ which is equivalent to
		\[
		R_{\varphi}(u,v):=\frac{\varphi(u)-\varphi(v)}{u-v},\quad u,v\in[0,R),
		\]
		being increasing in $u$ for fixed $v$ and vice versa
		(since $R_{\varphi}$ is symmetric). 
		
		Hence, since $s\leq s'(r)$ and $\widetilde{s}\leq\widetilde{s}'(r)$, we obtain
		\begin{align*}
		\frac{\varphi(s'(r))-\varphi(\widetilde{s}'(r))}{s'(r)-\widetilde{s}'(r)} & \geq\frac{\varphi(s)-\varphi(\widetilde{s})}{s-\widetilde{s}}\\
		& =\frac{(\varphi(s)-\log r)-(\varphi(\widetilde{s})-\log r)}{s-\widetilde{s}}=\frac{\varphi(s'(r))-\varphi(\widetilde{s}'(r))}{s-\widetilde{s}}
		\end{align*}
		which implies
		\begin{equation}
		\label{eq:phi_convexity}
		|s'(r)-\widetilde{s}'(r)|\leq|s-\widetilde{s}|.
		\end{equation}
		\item Assume $r_{\min}(\widetilde{s})\leq r<r_{\min}(s)$: Here\footnote{This case only occurs if $\lim_{t \uparrow R}\varphi(t) = - \log \inf \rho < \infty$. In that situation define $\varphi(R) := - \log \inf \rho$ and observe that 
			with this extension $\varphi$ is increasing and convex on $[0,R]$.}
		\[
		\left|\varphi^{-1}(\varphi(s)-\log r)-\varphi^{-1}(\varphi(\widetilde{s})-\log r)\right|=|\widetilde{s}'(r)-R|.
		\]
		By the fact that $\varphi$ is increasing and convex it is continuous, such that
		there
		exists an $\hat{s}\in [0,R)$ with $s\leq\hat{s}\leq\widetilde{s}$
		satisfying
		\[
		-\log\inf\rho=\varphi(\hat{s})-\log r
		\]
		and, hence, $\hat{s}'(r)=R$. By employing the same reasoning as in~(\ref{eq:phi_convexity}) using the convexity of $\varphi$ we have that 
		\[
		|\widetilde{s}'(r)-R|\leq|\widetilde{s}-\hat{s}|\leq|s-\widetilde{s}|.
		\]
	\end{enumerate}
	
	This finishes the proof.
\end{proof}

It is fair to ask whether the estimate can be improved. The following example 
answers this question. Namely, in any dimension we find a parameterized family of unnormalized densities for which \eqref{eq:Wasscontr} holds with equality.

\begin{example}
	Let $\alpha>0$ be an arbitrary parameter.
	With the notation of Theorem~\ref{thm:WD_log-concave} set $R=\infty$ and $\varphi(s)=\alpha s$ on $[0,\infty)$. The function $\varphi$ is strictly increasing and concave on $[0,\infty)$. Hence, for $\rho \colon \mathbb{R}^d \to (0,\infty)$ with $\rho(x) = \exp(-\alpha\vert x\vert)$ 
	the estimate of \eqref{eq:Wasscontr} is true. 
	Further observe that $G(t) = B^{(d)}_{(\log t^{-1})/\alpha}$.
	For $x,y\in\mathbb{R}^d$ we use again the Kantorovich-Rubinstein duality formula of the Wasserstein distance 
	w.r.t.~$U_\rho(x,\cdot)$ and $U_\rho(y,\cdot)$, that is,
	\begin{align}  \label{KanRub}
	W(U_{\rho}(x,\cdot),U_{\rho}(y,\cdot)) & =\underset{\left\Vert g\right\Vert _{{\rm Lip}}\leq1}{\sup}\left|\int_{\mathbb{R}^d} g(z)\left(U_{\rho}(x,{\rm d}z)-U_{\rho}(y,{\rm d}z)\right)\right|,
	\end{align}
	where $\Vert g \Vert_{\rm Lip} := \sup_{x,y\in\mathbb{R}^d} \frac{\vert g(x)-g(y) \vert}{\vert x- y \vert }$ for $g\colon \mathbb{R}^d \to \mathbb{R}$. 
	We argue as in Remark~\ref{rem: opt_coupl_ust} and set
	$h(z)=\vert z \vert$. Note that this function satisfies $\Vert h \Vert _{\rm Lip} \leq 1$ as well as
	\begin{align*}
	\int_{\mathbb{R}^d} h(z) U_\rho(x,{\rm d} z) &  = 
	\frac{1}{\rho(x)} \int_0^{\rho(x)} \int_{B^{(d)}_{(\log t^{-1})/\alpha}} \vert z \vert \frac{{\rm d} z}{\lambda_d\big(B^{(d)}_{(\log t^{-1})/\alpha}\big)} {\rm d}t \\
	& = \frac{1}{\rho(x)} \int_0^{\rho(x)} \int_{B^{(d)}_{1}} \frac{\log t^{-1}}{\alpha}\cdot \vert z \vert \frac{{\rm d} z}{\lambda_d\big(B^{(d)}_{1}\big)} {\rm d}t \\ 
	& = \frac{d}{(d+1)\alpha} \cdot \frac{1}{\rho(x)} \int_0^{\rho(x)} \log t^{-1} {\rm d} t
	= \frac{d}{(d+1)\alpha} \left( -\log\rho(x)-1 \right)\\
	& = \frac{d}{(d+1)\alpha} \left( \alpha \vert x \vert  -1 \right),
	\end{align*}
	where we again used the fact that
	$\int_{B^{(d)}_{1}}|z|\textrm{d}z=\frac{d}{d+1}\lambda_d\big(B^{(d)}_{1}\big).$
	Hence, by \eqref{KanRub}, employing the function $h$ we get a lower bound of $W(U_{\rho}(x,\cdot),U_{\rho}(y,\cdot))$,
	which coincides with the upper bound \eqref{eq:Wasscontr}.
	Thus, the Markovian coupling $c(x,y,\cdot)\in \Gamma(U_\rho(x,\cdot),U_\rho(y,\cdot))$
	constructed in Lemma~\ref{lema:WD-1}
	is in this scenario optimal and 
	\[
	W(U_{\rho}(x,\cdot),U_{\rho}(y,\cdot)) =\left(1-\frac{1}{d+1}\right)\big\vert \vert x\vert - \vert y\vert \big\vert, \qquad x,y\in \mathbb{R}^d.
	\]
	This establishes that the inequality stated in Theorem~\ref{thm:WD_log-concave} can, in general, not be improved.
\end{example}

\section{Spectral gap estimate}
\label{sec:Spectral-gap}
In this section we investigate spectral gap properties of the Markov operator induced by the transition kernel $U_\rho$ of the Markov chain $(X_n)_{n\in\mathbb{N}}$. For this we need further definitions. By $L_2(\pi)$ we denote the Hilbert space of functions $f\colon G\to \mathbb{R}$ with finite norm $\Vert f \Vert_{2,\pi} := \left( \int_G \vert f \vert^2 {\rm d}\pi \right)^{1/2}$. By the reversibility of $U_\rho$ we have that $\pi$ is a stationary distribution. The transition kernel $U_\rho$ can be extended to a linear operator
${U}_{\rho}\colon L_{2}(\pi)\to L_{2}(\pi)$ defined by 
\[
U_{\rho}f(x):=\int_{G}f(y)U_{\rho}(x,{\rm d}y), \qquad x\in G.
\]
It is well known that a general Markov operator is self-adjoint on $L_2(\pi)$ iff the corresponding transition kernel is reversible w.r.t.~$\pi$, see for example \cite[Lemma~3.9]{Ru12}. 
We denote the (mean) functional $\mathbb{E}_\pi \colon L_2(\pi) \to \mathbb{R}$ by $\mathbb{E}_\pi(f):= \int_G f {\rm d}\pi$ and note that this can be extended to a bounded linear operator $\mathbb{E}_\pi \colon L_2(\pi) \to L_2(\pi)$ 
with $\mathbb{E}_\pi(f) \equiv \int_G f {\rm d}\pi$. 
With this notation the spectral gap of $U_{\rho}$ is determined by the operator norm of $U_\rho - \mathbb{E}_\pi$, i.e., it is given by 
\[
{\rm gap}_\pi (U_{\rho}):=1-\left\Vert U_{\rho}-\mathbb{E}_{\pi}\right\Vert _{L_{2}(\pi)\to L_{2}(\pi)}.
\]
Further let $L^0_2(\pi)$ be the set of functions $f\in L_2(\pi)$ with $\mathbb{E}_\pi(f)=0$. Using the normed linear space $L_2^0(\pi)$ it is well known that $\Vert U_\rho \Vert_{L_2^0(\pi)\to L_2^0(\pi)}=\Vert U_\rho - \mathbb{E}_\pi\Vert _{L_{2}(\pi)\to L_{2}(\pi)}$, see e.g. \cite[Lemma~3.16]{Ru12}, such that
\[
{\rm gap}_\pi (U_{\rho})=1-\left\Vert U_{\rho}\right\Vert _{L^0_{2}(\pi)\to L^0_{2}(\pi)}.
\]
An immediate consequence of 
Theorem~\ref{thm:WD_log-concave}, for example by applying \cite[Proposition~30]{Ol09}, is the following:
\begin{cor}  \label{cor:spec_gap_from_Wass}
	Assume that $\varphi$ satisfies the conditions formulated in Theorem~\ref{thm:WD_log-concave} and
	$\rho(x)=\exp(-\varphi(\vert x\vert))$. Then 
	\[
	{\rm gap}_\pi (U_{\rho}) \geq \frac{1}{d+1}.
	\]
\end{cor}

The aim of this section is to extend and improve the previous estimate to a larger class of density functions which are not necessarily log-concave and rotational invariant.

For this, in addition to the Markov chain $(X_n)_{n\in\mathbb{N}}$, the auxiliary variable Markov chain $(T_n)_{n\in\mathbb{N}}$
also determined by Algorithm~\ref{alg:simple_slice} is useful. In the next section we introduce the corresponding transition kernel, provide a relation to $U_\rho$ and investigate further properties of $(T_n)_{n\in\mathbb{N}}$.

\subsection{Auxiliary variable Markov chain}

The sequence of auxiliary random variables $(T_{n})_{n\in\mathbb{N}}$ from
Algorithm~\ref{alg:simple_slice} provides also a Markov chain. In contrast to $(X_n)_{n\in\mathbb{N}}$ the Markov chain $(T_n)_{n\in\mathbb{N}}$ is defined on $(\mathbb{R}^+,\mathcal{B}(\mathbb{R}^+))$, with $\mathbb{R}^+:= (0,\infty)$ and
the transition kernel is given by 
\[
Q_{\rho}(t,B)=\frac{1}{\lambda_{d}(G(t))}\int_{G(t)}\frac{\lambda_{1}\left(B\cap[0,\rho(x)]\right)}{\rho(x)}{\rm d}x,\qquad B\in\mathcal{B}(\mathbb{R}^+).
\]
Recall that the level-set function of $\rho$ is given by $\ell_\rho(t) = \lambda_d(G(t))$ and define
a probability measure $\mu$ on $(\mathbb{R}^+,\mathcal{B}(\mathbb{R}^+))$
by
\[
\mu(B):=\frac{\int_{B}\ell_\rho(t){\rm d}t}{\int_{0}^{\infty}\ell_\rho(r){\rm d}r},\qquad B\in\mathcal{B}(\mathbb{R}^+).
\]
From \cite[Lemma~1]{latuszynski2014convergence} it follows that the transition kernel $Q_\rho$ is reversible w.r.t.~$\mu$. For the convenience of the reader we prove this fact in our setting. 
\begin{lem}
	The transition kernel $Q_{\rho}$ on $(\mathbb{R}^+,\mathcal{B}(\mathbb{R}^+))$ is reversible w.r.t.~$\mu$.
\end{lem}
\begin{proof}
	For any $A,B\in\mathcal{B}(\mathbb{R}^+)$ we have
	\begin{align*}
	\int_{B}Q_{\rho}(t,A)\mu({\rm d}t) & =\int_{B}\frac{1}{\lambda_{d}(G(t))}\int_{G(t)}\frac{\lambda_{1}\left(A\cap[0,\rho(x)]\right)}{\rho(x)}{\rm d}x\,\frac{\ell_\rho(t){\rm d}t}{\int_{0}^{\infty}\ell_\rho(r){\rm d}r}\\
	& =\int_{0}^{\infty}\boldsymbol{1}_{B}(t)\int_{G}\frac{\mathbf{1}_{G(t)}(x)}{\rho(x)}\int_{0}^{\infty}\boldsymbol{1}_{A}(s)\boldsymbol{1}_{[0,\rho(x)]}(s){\rm d}s\,\frac{{\rm d}x\,{\rm d}t}{\int_{0}^{\infty}\lambda_{d}(G(r)){\rm d}r}.
	\end{align*}
	Using the fact that $\mathbf{1}_{G(s)}(x)=\mathbf{1}_{[0,\rho(x)]}(s)$
	we have
	\begin{align*}
	\int_{B}Q_{\rho}(t,A)\mu({\rm d}t)=\int_{0}^{\infty}\int_{G}\int_{0}^{\infty}\boldsymbol{1}_{A}(s)\boldsymbol{1}_{B}(t)\frac{\mathbf{1}_{G(t)}(x) \mathbf{1}_{G(s)}(x)}{\rho(x)}\frac{{\rm d}s\,{\rm d}x\,{\rm d}t}{\int_{0}^{\infty}\lambda_{d}(G(r)){\rm d}r}.
	\end{align*}
	Note that the right-hand side of the previous equation is symmetric
	in $A$ and $B$, such that we can change their roles and argue backwards. This leads to
	\[
	\int_{B}Q_{\rho}(t,A)\mu({\rm d}t)
	=\int_{A}Q_{\rho}(t,B)\mu({\rm d}t),
	\]
	which finishes the proof.
\end{proof}
Now we present a relation of the spectral gap of $U_\rho$ to the spectral gap of $Q_\rho$. Here we need
the Hilbert space $L_2(\mu)$, which consists of functions $h\colon \mathbb{R}^+ \to \mathbb{R}$ 
with finite $\Vert h \Vert_{2,\mu}:= \left(\int_{\mathbb{R}^+} \vert h \vert^2 \mu({\rm d}t) \right)^{1/2}$. To state the spectral gap of $Q_\rho$ let $\mathbb{E}_\mu \colon L_2(\mu) \to \mathbb{R}$
be the (mean) functional given by $\mathbb{E}_\mu h := \int_{\mathbb{R}^+} h {\rm d} \mu$, which we consider as linear operator mapping $L_2(\mu)$ functions to constant ones. Then, the spectral gap
of $Q_\rho$ is given by the operator norm
\[
{\rm gap}_\mu (Q_\rho):= 1-\Vert Q_\rho - \mathbb{E}_\mu \Vert_{L_2(\mu)\to L_2(\mu)},
\]
where the transition kernel $Q_\rho$ is extended to the self-adjoint Markov operator $Q_\rho \colon L_2(\mu) \to L_2(\mu)$ defined by
\[
Q_\rho h(t):= \int_{\mathbb{R}^+} h(s) Q_\rho(t,{\rm d}s),
\qquad
t \in \mathbb{R}^+.
\]
Note that the self-adjointness here comes (again as for $U_\rho$) by the fact that $Q_\rho$ is reversible.
With this notation we obtain:
\begin{lem}
	\label{lem:equal_spectral_gap}The spectral gaps of $Q_\rho$ and $U_\rho$ coincide, that is,
	${\rm gap}_\pi(U_{\rho})={\rm gap}_\mu(Q_{\rho}).$
\end{lem}
\begin{proof}
	Define the linear operators $V\colon L_{2}(\mu)\to L_{2}(\pi)$ and $V^{*}\colon L_{2}(\pi)\to L_{2}(\mu)$
	by 
	\begin{align*}
	(Vg)(x) & :=\frac{1}{\rho(x)}\int_{0}^{\rho(x)}g(t){\rm d}t,\quad g\in L_{2}(\mu),\\
	(V^{*}f)(t) & :=\frac{1}{\lambda_{d}(G(t))}\int_{G(t)}f(x){\rm d}x,\quad f\in L_{2}(\pi).
	\end{align*}
	Now we show that $V^{*}$ is the adjoint operator of $V$,
	i.e., $\langle Vg,f\rangle_{\pi}=\langle g,V^{*}f\rangle_{\mu}$, where
	$\langle\cdot,\cdot\rangle_{\pi}$ and $\langle\cdot,\cdot\rangle_{\mu}$
	are the inner products of $L_{2}(\pi)$ and $L_{2}(\mu)$, respectively.
	We have
	\begin{align*}
	\langle Vg,f\rangle_{\pi} & =\int_{G}(Vg)(x)f(x)\pi({\rm d}x)
	=\int_{G}\frac{1}{\rho(x)}\int_{0}^{\rho(x)}g(t){\rm d}t\,f(x)\frac{\rho(x)}{\int_{G}\rho(y){\rm d}y}{\rm d}x\\
	& =\int_{G}\int_{0}^{\infty}\boldsymbol{1}_{[0,\rho(x)]}(t)g(t)f(x){\rm d}t\frac{{\rm d}x}{\int_{G}\rho(y){\rm d}y}.
	\end{align*}
	Further we use the fact that $\boldsymbol{1}_{[0,\rho(x)]}(t)=\boldsymbol{1}_{G(t)}(x)$, that $\int_{G}\rho(y){\rm d}y=\int_{0}^{\infty}\ell_\rho(r){\rm d}r$
	and change the order of the integrals. Finally, we have
	\begin{align*}
	\langle Vg,f\rangle_{\pi} & =\int_{0}^{\infty}g(t)\int_{G}f(x)\boldsymbol{1}_{G(t)}(x){\rm d}x\frac{{\rm d}t}{\int_{0}^{\infty}\ell_\rho(r){\rm d}r}\\
	& =\int_{0}^{\infty}g(t)\frac{1}{\lambda_{d}(G(t))}\int_{G(t)}f(x){\rm d}x\frac{\ell_\rho(t){\rm d}t}{\int_{0}^{\infty}\ell_\rho(r){\rm d}r}\\
	& =\int_{0}^{\infty}g(t)(V^{*}f)(t)\mu({\rm d}t)=\langle g,V^{*}f\rangle_{\mu}.
	\end{align*}
	Furthermore, we have $U_{\rho}=VV^{*}$ and $Q_{\rho}=V^{*}V$. Now,
	define $S\colon L_{2}(\mu)\to L_{2}(\pi)$ and $S^{*}\colon L_{2}(\pi)\to L_{2}(\mu)$
	by 
	\begin{align*}
	S(g) & :=\int_{0}^{\infty}g(t)\mu({\rm d}t),\quad g\in L_{2}(\mu),\\
	S^{*}(f) & :=\int_{G}f(x)\pi({\rm d}x),\quad f\in L_{2}(\pi).
	\end{align*}
	Also, note here that $S^{*}$ is the adjoint operator of $S$, as
	well as, $\mathbb{E}_{\pi}=SS^{*}$ and $\mathbb{E}_{\mu}=S^{*}S$. Define
	$R:=V-S$ and the adjoint $R^{*}=V^{*}-S^{*}$. By the fact that also $\mathbb{E}_{\pi}=SV^{*}=VS^{*}$
	we have 
	\begin{align*}
	RR^{*}=(V-S)(V^{*}-S^{*})=VV^{*}-\mathbb{E}_{\pi}=U_\rho-\mathbb{E}_{\pi}.
	\end{align*}
	Similarly, by $ \mathbb{E}_\mu = V^* S = S^* V$ we obtain $R^{*}R=Q_\rho-\mathbb{E}_{\mu}$. Now using the well-known fact, see e.g. \cite[Proposition~2.7]{Co85}, that
	\begin{align*}
	\left\Vert R \right\Vert_{L_{2}(\mu)\to L_{2}(\pi)} = \left\Vert R^{*} \right\Vert_{L_{2}(\pi)\to L_{2}(\mu)}
	\end{align*}
	the statement of the lemma
	follows by
	\begin{align*}
	\left\Vert RR^{*}\right\Vert _{L_{2}(\pi)\to L_{2}(\pi)}= \left\Vert R\right\Vert _{L_{2}(\mu)\to L_{2}(\pi)}^2 = \left\Vert R^{*}\right\Vert _{L_{2}(\pi)\to L_{2}(\mu)}^2 = \left\Vert R^{*}R\right\Vert _{L_{2}(\mu)\to L_{2}(\mu)}
	\end{align*}
	and the definition
	of the spectral gap.
\end{proof}

\begin{rem}
	Similar arguments as in the previous proof have been used in \cite[Section~4.2]{Ul14} in a finite state space setting as well as in \cite{latuszynski2014convergence,RuUl13,RuUl15}.
\end{rem}

Now we argue that the transition kernel $Q_\rho$ (and therefore also the Markov operator) only depends on $\rho$ via its level-set function
$\ell_\rho$. 
\begin{lem}
	\label{lem:Q_depends_on_l}
	For an unnormalized density $\rho \colon G \to \mathbb{R}^+$ 
	we have for any $t\in\mathbb{R}^+$ that
	\begin{align*}
	Q_{\rho}(t,B)=\frac{1}{\ell_\rho(t)}\int_{t}^{\infty}\frac{\lambda_{1}\left(B\cap\left[0,r\right]\right)}{r}{\rm d}(-\ell_\rho)(r), \quad B\in\mathcal{B}(\mathbb{R}^+),
	\end{align*}
	where on the right-hand side we use the Lebesgue-Stieltjes integral w.r.t.~$-\ell_\rho$.
\end{lem}
\begin{proof}
	Let $g\colon (t,\ell_\rho(0)) \to \mathbb{R}^+$ with
	$
	g(r)= \lambda_{1}\left(B\cap\left[0,r\right]\right)/r
	$
	and note that the pushforward measure $\rho_*\lambda_d$ on $\mathbb R_+$
	is defined by 
	\[
	\rho_*\lambda_d(B) := \lambda_{d}\circ\rho^{-1}(B)=\lambda_{d}\left(\rho^{-1}(B)\right),
	\qquad B\in\mathcal{B}(\mathbb{R}^+).
	\]
	Hence for any $r,s\in \mathbb{R^+}$ with $r<s$ we have
	\begin{align*}
	\rho_*\lambda_{d}((r,s])
	& =\lambda_{d}\left(\left\{ x\in G(t):r<\rho(x)\leq s\right\} \right)\\
	& =\lambda_{d}\left(\left\{ x\in G(t):r<\rho(x)\right\} \right) - \lambda_{d}\left(\left\{ x\in G(t):s<\rho(x)\right\} \right)\\
	& =-\left(\ell_\rho(s+)-\ell_\rho(r+)\right),
	\end{align*}
	where $\ell_\rho(t+)$ denotes the right limit at $t\in\mathbb R^+$ of the left-continuous level-set function.
	Thus, $\rho_*\lambda_{d}$ is the Lebesgue-Stieltjes measure associated to the monotone non-decreasing function $-\ell_\rho\colon \mathbb R_+ \to (-\infty, 0]$, see, e.g., \cite[Section 1.3.2]{AL06}, and we obtain with a change of variable, see \cite[Theorem 3.6.1, p. 190]{Bog07}, that
	\begin{align*}
	Q_{\rho}(t,B) & =\frac{1}{\ell_\rho(t)}\int_{G(t)}\frac{\lambda_{1}\left(B\cap[0,\rho(x)]\right)}{\rho(x)}{\rm d}x \\
	&= \frac{1}{\ell_\rho(t)} \int_{G(t)}g(\rho(x))\lambda_{d}({\rm d}x) \\
	&= \frac{1}{\ell_\rho(t)} \int_{t}^{\ell_\rho(0)} g(r)\; \rho_*\lambda_{d}({\rm d}r) \\
	&=\frac{1}{\ell_\rho(t)} \int_{t}^{\infty}\frac{\lambda_{1}\left(B\cap\left[0,r\right]\right)}{r}{\rm d}(-\ell_\rho)(r).
	\end{align*}
\end{proof}

\begin{rem}
	For a given $\rho  \colon G \to \mathbb{R}^+$ with continuously differentiable level-set function $\ell_\rho$ the previous result can be stated as
	\begin{align*}
	Q_{\rho}(t,B)=-\frac{1}{\ell_\rho(t)}\int_{t}^{\infty}\frac{\lambda_{1}\left(B\cap\left[0,r\right]\right)}{r}\;\ell_\rho'(r){\rm d} r , \quad B\in\mathcal{B}(\mathbb{R}^+).
	\end{align*}
\end{rem}

An immediate consequence of Lemma~\ref{lem:equal_spectral_gap} and Lemma~\ref{lem:Q_depends_on_l} is the following important result.

\begin{cor}
	\label{cor:dif}
	Let $d,\widetilde{d} \in \mathbb{N}$ and $ G\subseteq \mathbb{R}^d$ as well as $\widetilde  G\subseteq \mathbb{R}^{\widetilde{d}}$. Further let $\rho \colon G \to \mathbb{R}^+$ and $\widetilde \rho \colon \widetilde G \to \mathbb{R}^+$ satisfying $\ell_\rho(t) = \ell_{\widetilde \rho}(t)$ for all $t\in\mathbb{R}^+$.
	Then
	\[
	Q_{\rho}(t,B) = Q_{\widetilde{\rho}}(t,B),\qquad t\in\mathbb{R}^+,\; B\in\mathcal{B}(\mathbb{R}^+).
	\]
	and
	\[
	{\rm gap}_\pi(U_{\rho})={\rm gap}_\mu(Q_{\rho})={\rm gap}_\mu(Q_{\widetilde{\rho}})={\rm gap}_{\widetilde{\pi}}(U_{\widetilde{\rho}}),
	\]
	where $\widetilde{\pi}$ denotes the distribution induced by $\widetilde{\rho}$.
\end{cor}

Thus, the above corollary tells us that the spectral gap of simple slice sampling is entirely determined by the level-set function $\ell_\rho\colon \mathbb R^+ \to [0,\infty)$ of the (unnormalized) target density $\rho$ and does, for instance, not necessarily depend on the dimension of $G$.
In particular, Corollary \ref{cor:dif} allows us to extend the spectral gap result of Corollary \ref{cor:spec_gap_from_Wass} to much larger classes of target distributions as we explain in detail in the next subsection.

\subsection{Spectral gap result\label{subsec:Estimation-of-the-SG}}
Corollary \ref{cor:dif} implies that the lower bound for the spectral gap of simple slice sampling of rotational invariant and log-concave (unnormalized) target densities also holds for other target densities which share the same level-set function.
Thus, our idea is to identify convenient classes of target densities $\rho\colon G \to [0,\infty)$, with $G \subseteq \mathbb R^d$, which possess the same level-set function as a rotational invariant and log-concave unnormalized density $\widetilde \rho\colon \widetilde G \to [0,\infty)$, with $\widetilde G \subseteq \mathbb R^{\widetilde d}$.
We illustrate this approach first by an example and formalize it rigorously afterwards.

\begin{example} \label{ex:two_peak_ell}
	We consider a bimodal distribution $\pi$ on the set
	\[
	G=(m_0+\mathring{B}^{(d)}_{\sqrt{\log 16}})\cup \mathring{B}^{(d)}_{\sqrt{\log 4}}\subset \mathbb R^d
	\]
	with $m_0 = (5,0,\ldots,0)\in\mathbb{R}^d$ given by the unnormalized density
	\[
	\rho(x)=\max\bigg\{\exp\left(-\frac{1}{2}\left|x\right|^{2}\right),\exp\left(-\frac{1}{4}\left|x-m_{0}\right|^{2}\right)\bigg\}-\frac{1}{2}.
	\]
	Notice that $\rho$ is positive on $G$.
	Here it is worth to mention that in particular in such scenarios an efficient implementation of simple slice sampling is challenging and we are at this point merely interested in theoretical properties.
	By construction, 
	the level sets of $\rho$ consist of two disjoint balls, i.e., we have
	\[
	G(t) = \bigg(m_0+\mathring{B}^{(d)}_{\sqrt{\log (1/2+t)^{-4}}}\bigg)\cup \mathring{B}^{(d)}_{\sqrt{\log (1/2+t)^{-2}}}, \qquad t\in[0,1/2).
	\]
	This leads to
	\[
	\ell_\rho(t) = \big( 2^{d/2} + 4^{d/2} \big) \lambda_d(B^{(d)}_1) \big( \log(1/2+t)^{-1} \big)^{d/2},
	\qquad t\in[0,1/2). 
	\]
	In Figure~\ref{fig:two_peaks_plot} and Figure~\ref{fig:ell_two_peaks_plot} we provide an illustration of $\rho$ and $\ell_\rho$ for $d=2$.
	\begin{figure}
		\begin{minipage}[t]{0.48\columnwidth}%
			\includegraphics[width=\columnwidth]{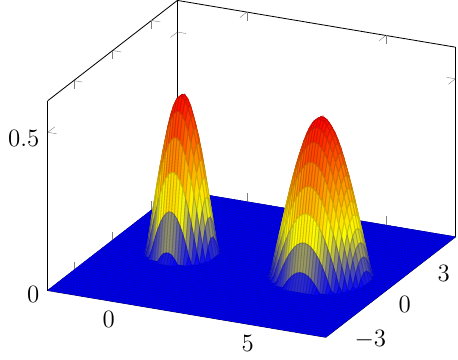}
			\centering{}\caption{Plot of $\rho$ from Example~\ref{ex:two_peak_ell} for $d=2$.}
			\label{fig:two_peaks_plot}
		\end{minipage}\hfill{}%
		\begin{minipage}[t]{0.48\columnwidth}
			\includegraphics[width=\columnwidth]{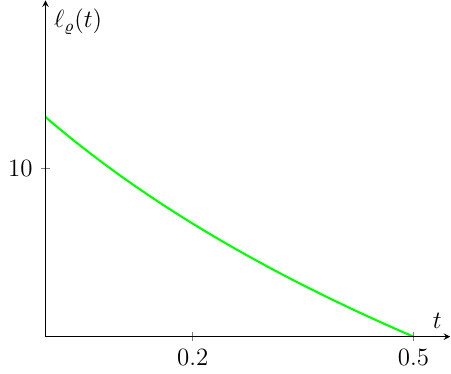}
			\caption{Plot of $\ell_\rho$ of Example~\ref{ex:two_peak_ell} for $d=2$.}%
			\label{fig:ell_two_peaks_plot}
		\end{minipage}
	\end{figure}
	Straightforwardly one obtains the inverse of $\ell_\rho$ given by $\ell_\rho^{-1}\colon (0,\ell_\rho(0)) \to (0,1/2)$ with
	\[
	\ell^{-1}_\rho(s) = \exp\bigg(-\bigg( \frac{s}{(2^{d/2}+2^d)\lambda_d\big( B^{(d)}_1 \big)} \bigg)^{2/d}\bigg) - 1/2.
	\]
	Now, for $k\in\mathbb{N}$ 
	we can define rotational invariant unnormalized densities
	\[
	\widetilde \rho^{(k)} \colon B^{(k)}_{(\ell_\rho(0)/\lambda_k(B^{(k)}_1))^{1/k}} \to (0,\infty)
	\]
	by
	\[
	\widetilde \rho^{(k)}(y) := \ell_\rho^{-1} (\lambda_k(B^{(k)}_1)\vert y \vert^k)
	\]
	which have the same level-set function as $\rho$, i.e., $\ell_\rho(t)=\ell_{\widetilde \rho^{(k)}}(t)$ for all $t\in(0,1/2)$. 
	Note that the dimension of the domain of $\widetilde \rho^{(k)}$ is $k$, whereas for $\rho$ it is $d$ and $d$ does not need to coincide with $k$. 
	In Figure~\ref{fig:tilde_rho_k1} and Figure~\ref{fig:tilde_rho_k2} we display $\widetilde \rho^{(k)}$ for $k=1$, $k=2$ and $d=2$.
	By Corollary~\ref{cor:dif} we can conclude that the spectral gaps of $U_\rho$ and $U_{\widetilde \rho^{(k)}}$ are the same. 
	Moreover, the auxiliary densities $\widetilde \rho^{(k)}$ are of the form $\widetilde \rho^{(k)}(x)=\exp(- \varphi_k(|x|))$ on their domain, where 
	\[
	\varphi_k(s):=-\log \ell^{-1}(s^k)=-\log \bigg(\exp\bigg(-\bigg( \frac{s^k}{(2^{d/2}+2^d)\lambda_d\big( B^{(d)}_1 \big)} \bigg)^{2/d}\bigg) - 1/2\bigg)
	\]
	for all $s\in \big[0,(\ell_\rho(0)/\lambda_k(B^{(k)}_1)^{1/k}\big)$. Thus, for $k \geq \lceil \frac d2\rceil$ the function $\varphi_k$ is strictly increasing and convex, i.e., the unnormalized density $\widetilde \rho^{(k)}$ satisfies the assumptions of Theorem \ref{thm:WD_log-concave} and Corollary \ref{cor:spec_gap_from_Wass}, respectively.
	Hence, we can conclude that simple slice sampling of the bimodal target $\pi$ on $\mathbb R^d$ given by $\rho$ has a spectral gap of at least
	\[
	{\rm gap}_\pi\left(U_{\rho}\right)\geq\frac{1}{\lceil \frac d2\rceil+1}.
	\]

	\begin{figure}
		\begin{minipage}[t]{0.48\columnwidth}%
			\includegraphics[width=\columnwidth]{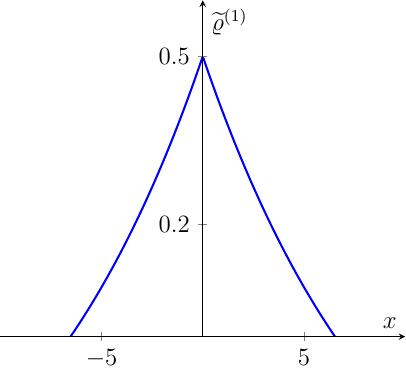}
			\centering{}\caption{Plot of $\widetilde \rho^{(1)}$.}
			\label{fig:tilde_rho_k1}
		\end{minipage}\hfill{}%
		\begin{minipage}[t]{0.48\columnwidth}
			\includegraphics[width=\columnwidth]{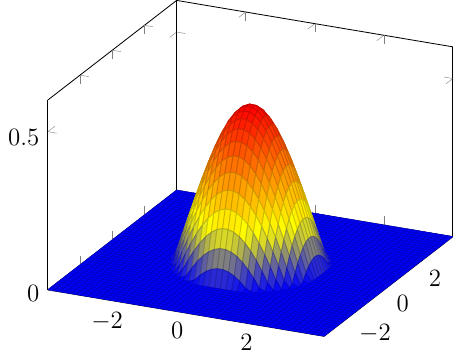}
			\caption{Plot of $\widetilde \rho^{(2)}$.}%
			\label{fig:tilde_rho_k2}
		\end{minipage}
	\end{figure}
\end{example}

The previous example suggests the definition of the following classes of level-set functions.

\begin{defi}
	\label{def:exp}
	A continuous function $\ell\colon (0,\infty) \to [0,\infty]$ belongs to the class $\Lambda_k$ with $k\in\mathbb{N}$ if
\end{defi}
\begin{enumerate}
	\item $\ell$ is strictly decreasing on its open support 
	\[
	\supp \ell :=\left(0,\sup\big\{t\in(0,\infty)\mid \ell(t)> 0 \big\}\right),
	\]
	which implies the existence of the inverse $\ell^{-1}$ on $\ell(\supp \ell) = (0,  \|\ell\|_\infty)$ with
	\[
	\|\ell\|_\infty := \sup_{s \in (0,\infty)
	} = \lim_{t\to 0+} \ell(t) = \ell(0+),
	\]
	\item the function $g\colon \big( 0,  \|\ell\|_\infty^{1/k} \big) \to \supp \ell$ given by $g(s) := \ell^{-1}(s^k)$ 
	is log-concave, that is, $\log g$ is concave.
\end{enumerate}

The main result of this section is then as follows:

\begin{thm}
	\label{thm:exp}
	For an unnormalized density $\rho\colon G\to \mathbb{R}^+$ assume that its level-set function $\ell_\rho \in \Lambda_k$ for $k\in\mathbb{N}$. Then
	\[
	{\rm gap}_\pi\left(U_{\rho}\right)\geq\frac{1}{k+1}.
	\]
\end{thm}
\begin{proof}
	The idea here is to construct an unnormalized density $\widetilde \rho^{(k)} \colon \mathbb{R}^k \to  \mathbb{R}^+$ such that $\ell_{\rho} = \ell_{\widetilde \rho^{(k)}}$ and $\widetilde \rho^{(k)}$ satisfies the assumptions of Theorem~\ref{thm:WD_log-concave}. 
	The statement then follows by Corollary~\ref{cor:spec_gap_from_Wass} and Corollary \ref{cor:dif}. 
	To this end, we define $\widetilde{\rho}^{(k)} \colon \mathring{B}_{R_k}^{(k)} \to \mathbb{R}^+$ with ${R_k := \left(\Vert \ell_\rho \Vert_\infty / \lambda_k(B_1^{(k)}) \right)^{1/k}}$ by
	\[
	\widetilde{\rho}^{(k)} (x) := \ell_\rho^{-1}\left(\lambda_k(B_1^{(k)})\vert x \vert^k \right),
	\qquad
	|x| < R_k.
	\]
	By construction we have for any $t\in (0,\infty)$
	\[
	\ell_{\widetilde{\rho}^{(k)}}(t) = 
	\lambda_{k}\left(\left\{ x\in\mathbb{R}^{k}:\widetilde\rho^{(k)}(x)\geq t\right\} \right) =\ell_\rho(t).
	\]
	Next, we observe that $\widetilde \rho^{(k)}(x)=\exp(-\varphi_k\left(\left|x\right|\right))$ for $|x|<R_k$ with
	\[
	\varphi_k(s):=-\log\ell_\rho^{-1}\left(\lambda_k(B_1^{(k)}) s^{k}\right),
	\qquad
	s \in [0, R_k).
	\]
	Since $\ell_\rho$ belongs to $\Lambda_k$, we know that $s \mapsto \log \ell_\rho^{-1}\left(s^{k}\right)$ is concave.
	This yields the convexity of $\varphi_k$ on $[0, R_k)$.
	Moreover, $\ell_\rho \in \Lambda_k$ implies that also $\ell_\rho^{-1}$ is strictly decreasing on $[0, \|\ell_\rho\|_\infty)$.
	Thus, the mapping $s \mapsto \log \ell_\rho^{-1}\left(s^{k}\right)$ is strictly decreasing and, therefore, $\varphi_k$ is strictly increasing.
	Hence, the unnormalized density $\widetilde{\rho}^{(k)}$ satisfies the assumptions of Theorem~\ref{thm:WD_log-concave} which finishes the proof.
\end{proof}

Notice that the lower the number $k$ of the class $\Lambda_{k}$ the larger the lower bound of the spectral gap. 
Subsequently, we provide some (sufficient) characterizations of the classes $\Lambda_k$.

\subsubsection{Properties of the class $\Lambda_{k}$}

The requirements of a level-set function to belong to the class $\Lambda_{k}$ are not easy to check. We provide some auxiliary tools. The following is a trivial consequence of the definition of $\Lambda_k$.

\begin{prop} \label{prop:constant_ell}
	If $\ell\in\Lambda_k$ for $k\in\mathbb{N}$ and $c>0$, then $c\cdot \ell \in \Lambda_k$. 
\end{prop}

Now a sufficient condition for being in $\Lambda_1$ is stated.

\begin{prop}
	If $\ell\colon (0,\infty) \to [0,\infty)$ is strictly decreasing and concave, then $\ell\in\Lambda_{1}.$
\end{prop}
\begin{proof}
	Since $\ell$ is strictly decreasing and concave we have that $\ell^{-1}$ is concave. Then $\log\ell^{-1}$
	is log-concave and $\ell\in\Lambda_{1}$.
\end{proof}

Assuming smoothness of $\ell$ the previous result can be extended and provides a characterisation of $\Lambda_k$.

\begin{prop}\label{prop:L_k_property}
	Let $\ell\colon (0,\infty) \to [0,\infty)$ be continuously differentiable on its open support $\supp \ell$ with $\ell'(t)<0$.
	Define the function $\psi\colon \supp \ell \to [0,\infty)$ by $\psi(t):=\frac{t \ell'(t)}{\ell(t)^{1-1/k}}$ for $k\in\mathbb{N}$. Then
	\[
	\ell \in \Lambda_k \quad\Longleftrightarrow\quad \psi \text{  is decreasing}.
	\]
\end{prop}
\begin{proof}
	The function $\ell$ is strictly decreasing on $\supp \ell$, since $\ell'(t)<0$ on that interval.
	This implies that the inverse $\ell^{-1}\colon [0, \|\ell\|_\infty) \to \supp \ell$ exists and is strictly decreasing.
	Define the function $\varphi_k\colon [0, \|\ell\|_\infty^{1/k}) \to \mathbb{R}$ with 
	$\varphi_{k}(s):=-\log\ell^{-1}(s^{k})$. 
	Observe that $\varphi_k$ is strictly increasing and by the inverse mapping theorem continuously differentiable on $\supp \ell$.
	We have 
	\begin{align*}
	\varphi_{k}'(s) & =-\frac{{\rm d}}{{\rm d}s}\log\ell^{-1}(s^{k})=-\frac{1}{\ell^{-1}(s^{k})}\ \left(\frac{{\rm d}}{{\rm d}s}\ell^{-1}(s^{k})\right)=-\frac{1}{\ell^{-1}(s^{k})}\ \frac{ks^{k-1}}{\ell'(\ell^{-1}(s^{k}))}.
	\end{align*}
	Given the assumptions we have that $\ell\in \Lambda_k$ if and only if $\varphi_k$ is convex.
	The latter is equivalent to $\varphi_{k}'$ being increasing. 
	Note that for $s \in  [0, \|\ell\|_\infty^{1/k})$
	\begin{align*}
	\frac{ks^{k-1}}{\ell'(\ell^{-1}(s^{k}))} & =k\frac{s^{k}}{s\ell'(\ell^{-1}(s^{k}))}=k\frac{\ell(\ell^{-1}(s^{k}))}{s\ell'(\ell^{-1}(s^{k}))}=k\frac{\ell(\ell^{-1}(s^{k}))}{\left(\ell(\ell^{-1}(s^{k}))\right)^{1/k}\ \ell'(\ell^{-1}(s^{k}))}\\
	& =k\frac{\left(\ell(\ell^{-1}(s^{k}))\right)^{1-1/k}}{\ell'(\ell^{-1}(s^{k}))}.
	\end{align*}
	Hence, with $h(t):=-\frac{\ell^{1-1/k}(t)}{t \ell'(t)}$ we obtain
	$
	\varphi_{k}'(s)=k\cdot h(\ell^{-1}(s^{k})),
	$
	which leads to the fact that
	\[
	\varphi_{k}'\text{ increasing}\quad\Longleftrightarrow\quad h\text{ decreasing}.
	\]
	However, the latter is equivalent to the fact that the mapping $t\mapsto\frac{t\ell'(t)}{\ell(t)^{1-1/k}}$
	is decreasing, since $\frac{\ell(t)^{1-1/k}}{t\ell'(t)}<0$ on $\supp \ell$.
\end{proof}
\begin{rem}
	Roberts and Rosenthal \cite{RoRo99} derived convergence results
	of simple slice sampling given the assumption that $t\mapsto t\ell'(t)$ is decreasing which corresponds to the sufficient condition for $\ell \in \Lambda_1$.
	In particular, they write ``However, it is surprising that
	this same bound\,\footnote{They provide a quantitative bound of $\Vert U_\rho^n(x,\cdot) -\pi(\cdot) \Vert_{\rm tv}$ for any continuously differentiable $\ell$ as in Proposition~\ref{prop:L_k_property}.} applies to any density $\rho$ such that $t \ell'(t)$
	is non-increasing''\footnote{For the formulas we adapted their statement to our notation, namely in their work our $\rho$ is $\pi$ and our $\ell$ is denoted by $Q$.}. 
	We also observe this surprising fact, but w.r.t. the spectral gap.
	In contrast to their result, in general,
	we do not require the existence of the first derivative from the level-set function. 
	Moreover, our result for $\Lambda_{k}$ with $k>1$ has no analogues in the work of Roberts and Rosenthal. 
	To emphasize this we consider in Section~\ref{sec:Examples}
	an example of a level-set function which is in $\Lambda_{2}$ but not in $\Lambda_{1}$.
\end{rem}

\subsubsection{Further examples\label{sec:Examples}}

We illustrate in two more examples
the advantages of Theorem~\ref{thm:exp} compared to Theorem~\ref{thm:WD_log-concave}.

\begin{example}  \label{ex:exponential}
	For $\alpha>0$ and $\gamma>0$ let $\rho^{(d)}\colon \mathbb{R}^d \to \mathbb{R}^+$ be given by $\rho^{(d)}(x) = \exp(-\alpha \vert x \vert ^\gamma)$. By Proposition~\ref{prop:constant_ell} it is sufficient to consider
	\[
	\ell(t):=\left(\frac{\log t^{-1}}{\alpha}\right)^{d/\gamma} = c_1 \, \ell_{\rho^{(d)}}(t),\qquad t\in(0,\infty),
	\]
	with $c_1 = \lambda_d(B_1^{(d)})$.	
	The function $\ell$ is strictly decreasing and $\log\ell^{-1}(s^{k})=-\alpha s^{\gamma\frac{k}{d}}$.
	Thus, for any $\gamma\geq1$ and $k=d$ it is concave on $(0,\infty)$, such that for this parameters
	$\ell\in\Lambda_{d}$ and by Theorem~\ref{thm:exp}
	\[
	{\rm gap}\left(U_{\rho^{(d)}}\right)\geq\frac{1}{d+1}.
	\]
	However, we notice that $\log\ell^{-1}(s^{k})=-\alpha s^{\gamma k/d}$ is concave for $k\geq \lceil d/\gamma \rceil$.
	Otherwise, for $k < \lceil d/\gamma \rceil$ it is convex. 
	Thus, we have that $\ell_\rho \in \Lambda_{\lceil d/\gamma \rceil}$ but if $d<\gamma$, then   $\ell_\rho \notin \Lambda_{\lceil d/\gamma \rceil-1}$.
	For instance, for $\gamma =d/2$ we have that $\ell_\rho \in \Lambda_{2}$ and $\ell_\rho \notin \Lambda_{1}$.
	Hence, 
	Theorem~\ref{thm:exp} tells us that for this class of target densities
	\[
	{\rm gap}\left(U_{\rho^{(d)}}\right)\geq\frac{1}{\left\lceil d/\gamma\right\rceil +1}\geq\frac{1}{d+1}.
	\] 
\end{example}

In the following we consider a ``volcano'' density.

\begin{example} \label{ex:volcano}
	Let $\rho^{(d)}\colon \mathbb{R}^d \to \mathbb{R}^+$ be given by  $\rho^{(d)}(x)=e^{-|x|^{2d}+2|x|^{d}}$.
	In contrast to Example~\ref{ex:exponential} here we have more than a single peak.
	For $d=2$ the density is plotted in Figure~\ref{fig: d2_volcano}.
	\begin{figure}
		\begin{minipage}[t]{0.46\columnwidth}
			\includegraphics[width=\columnwidth]{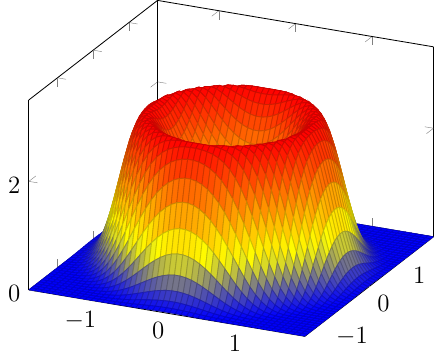}
			\caption{Plot of \mbox{$\rho^{(2)}(x)=e^{-|x|^{4}+2|x|^{2}}$.}}
			\label{fig: d2_volcano}
		\end{minipage}
		\hspace{0.08\columnwidth}
		\begin{minipage}[t]{0.45\columnwidth}%
			\includegraphics[width=\columnwidth]{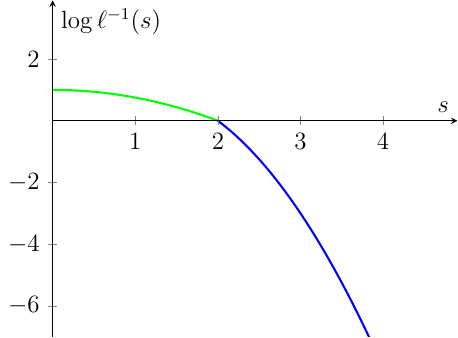}
			\caption{Plot of $\log \ell^{-1}(s)$ from Example~\ref{ex:volcano}. }
			\label{fig: log_ell_volcano}
		\end{minipage}\hfill{}%
	\end{figure}
	It is easy to see that $\ell_{\rho^{(d)}}$ is proportional to 
	the strictly decreasing function $\ell \colon (0,\infty)\to [0,\infty)$ given by
	\[
	\ell(t):=\begin{cases}
	1+\sqrt{1+\log t^{-1}}, & 0<t\leq1,\\
	2\sqrt{1+\log t^{-1}}, & 1<t\leq e,
	\end{cases}
	\]
	such that by Proposition~\ref{prop:constant_ell} it is sufficient to consider $\ell$.
	This leads to 
	\[
	\ell^{-1}(s)=\begin{cases}
	e^{1-\frac{s^{2}}{4}}, & 0\leq s<2,\\
	e^{-s^{2}+2s}, & s\geq2,
	\end{cases}
	\]
	and we have that $\log \ell^{-1}(s)$ is concave, see also Figure~\ref{fig: log_ell_volcano}. Hence $\ell\in \Lambda_1$ for arbitrary $d$ and Theorem~\ref{thm:exp} implies
	\[
	{\rm gap}_{\pi}(U_{\rho^{(d)}}) \geq \frac{1}{2}.
	\]
\end{example}

\section*{Acknowledgements}
Viacheslav Natarovskii thanks the DFG Research Training Group 2088 for their support.
Daniel Rudolf thanks Andreas Eberle for fruitful discussions on this topic and acknowledges support of the Felix-Bernstein-Institute for Mathematical Statistics in the Biosciences (Volkswagen Foundation) and the Campus laboratory AIMS. Bj\"orn Sprungk thanks the DFG for supporting this research within project 389483880 and the Research Training Group 2088.

%

\providecommand{\bysame}{\leavevmode\hbox to3em{\hrulefill}\thinspace}
\providecommand{\MR}{\relax\ifhmode\unskip\space\fi MR }
\providecommand{\MRhref}[2]{%
	\href{http://www.ams.org/mathscinet-getitem?mr=#1}{#2}
}
\providecommand{\href}[2]{#2}

\end{document}